\documentclass[12pt,a4paper]{article}

\usepackage{
	microtype,
	enumitem,
	multicol,
	multirow,
	linguex,
	amssymb,
	microtype,
	stmaryrd,
	amsmath,
	booktabs,
	xcolor,
	calculator,
	titling,
	changepage,
	float
}
\usepackage[numbers]{natbib}

\newcommand{\correspondingauthor}[1]{\textbf{Corresponding Author}: Ryan Mahendra Wilis, Essex, University of Essex, rw20018@essex.ac.uk}
\newtheorem{theorem}{Theorem}[section]
\newtheorem{lemma}[theorem]{Lemma}

\newtheorem{corollary}[theorem]{Corollary}

\newenvironment{proof}[1][Proof]{\begin{trivlist}
\item[\hskip \labelsep {\bfseries #1}]}{\end{trivlist}}

\newcommand{\qed}{\nobreak \ifvmode \relax \else
      \ifdim\lastskip<1.5em \hskip-\lastskip
      \hskip1.5em plus0em minus0.5em \fi \nobreak
      \vrule height0.75em width0.5em depth0.25em\fi}

\begin{document}

\thispagestyle{plain}




\vspace{2pt}
\begin{center}
    \LARGE{\textbf{An Algebraic Extension of the General Leibniz Product rule for  Fractional Indices and Applications}}
\end{center}

\begin{center}
\vspace{4pt}
\large
    Ryan Mahendra Wilis\textsuperscript{1} \footnote{\correspondingauthor{}}
    
\small
   \textsuperscript{1} University of Essex

\end{center}

\begin{small}
\begin{center}
\vspace{9pt}
\textbf{Abstract}    
\end{center}

\begin{adjustwidth}{20pt}{20pt}
\small \noindent 
This paper presents a reformulation of the Leibniz product rule as a finite sum that expresses the fractional derivative of the product of two differentiable functions. This paper then proves the cases for when the product consists of an arbitrary differentiable function and a Sheffer sequence , and the case for when the product consists of an arbitrary differentiable function and a confluent hypergeometric limiting function with a positive integer order.
Finally, this paper provides example expressions for certain Sheffer sequences: any Apell sequence, the falling factorial,  the rising factorial,  the exponential polynomial,  and the Associated Laguerre polynomial. 
 
\end{adjustwidth}

\end{small}

\title{An Umbral Extension of the General Leibniz Product rule for  Fractional Indices and Applications}
\author{Ryan Mahendra Wilis}

\section{Introduction}

The field of Fractional calculus has had some recent developments with respect to its theoretical framework and its applications (See:  \citep{fractional theory 1}, \citep{fractional theory 2}, \citep{fractional theory 3}, \citep{fractional theory 4}, \citep{fractional theory 5}, \citep{fractional theory 6}, \citep{fractional theory 7}). One of the main results is a new fractio-differential operator introduced by Antangana and Baleneau \citep{anataga} with a non-local and non-singular kernel. Recent research has also discussed investigating the fractional derivative of the product of certain special functions (See: \citep{product} and \citep{product 2}), and the fractional derivative of the certain classeses of Hypergeometric functins have also been discussed in recent literature (See: 
\citep{hypergeometric} and \citep{hypergeometric 2})

On the other hand, the study of Sheffer sequences , formally introduced and analysed by Roman \citep{Roman theory} through the theory of Umbral calculus, has had many developments with respect to its applications in many fields of mathematics, for instance in the study of p-adic invariant integrals and other topics.
(See: \citep{sheffer 1}, \citep{sheffer 2}, \citep{sheffer 3},
\citep{sheffer 4}).

This paper aims to reconcile the field of Fractional calculus and the study of Sheffer sequences to obtain new results pertaining the fractional derivatives of the products of  certain special functions. We analyse the case for when the product consists of an arbitrary differentiable function and a Sheffer sequence, and an arbitrary differentiable function and the Confluent hypergeometric limiting function.

\section{ The Truncated Leibniz product rule for non-integers}

Let us start with the General Leibniz rule of differentiation, where $f(x)$ and $g(x)$ are  n-times differentiable functions.

\begin{theorem}
\label{General Leibniz rule of differentiation}
For $n \in \mathbb{N}$ 
\[ (\partial_{x})^{n} [f(x)g(x)] = \sum_{m=0}^{n} \binom{n}{m} f^{(n-m)}(x) g^{(m)}(x)
\]
\end{theorem}
This theorem can be proven by the principle of mathematical induction and the product rule.

This section will provide and prove a truncated non-integer Leibniz product rule sum that expresses the fractional derivative of the product of two differentiable functions. The cases for when the product consists of an arbitrary differentiable function and a monomial with a positive interger order and a confluent hypergeometric limiting function with a positive integer order will be proven.

\subsection{The Fractional Leibniz product rule of an arbitrary differentiable function and a polynomial with a positive integer degree}

\begin{theorem}
\label{main formula 1}
For $n \in \mathbb{N}$, $a \in \mathbb{R}$
\[ (\partial_{x})^{a} [x^{n}f(x)]=  \sum_{m=0}^{n} \binom{n}{m} (a)_{m}  f^{(a-m)}(x) x^{(n-m)}
\]
\end{theorem}

\begin{proof}

 A simple combinatorial property concerning the binomial coefficent will be used to prove this theorem.

First, consider the $a$th derivative of $x^{n}f(x)$ and   make use of the extension of the General leibniz rule to non-integers.

\[ (\partial_{x})^{a} [x^{n}f(x)]= \sum_{m=0}^{a} \binom{a}{m} (n)_{m} f^{(a-m)}(x)  x^{n-m}
\]
\[
= 
\sum_{m=0}^{\infty} \binom{a}{m} (n)_{m} f^{(a-m)}(x)  x^{n-m}
\]

 A finite sum expression can be easily determined by noting that equation above is simply equal to 

\[ (\partial_{x})^{a} [x^{n}f(x)]= \sum_{m=0}^{\infty} \binom{n}{m} (a)_{m} f^{(a-m)}(x) x^{n-m}
\]

This can be noted by  switching the total number of elements of binomial coefficient from $a$ to $n$ of $\binom{a}{m}$ such that

\[ \binom{a}{m} (n)_{m} = \frac{(a)_{m} (n)_{m}}{m!} = \binom{n}{m} (a)_{m}
\]

Which, due to the nature of the binomial coefficient $\binom{n}{m}$, allows  truncation of the infinite sum to a finite sum with an upper index of $n$

\[ (\partial_{x})^{a} [x^{n}f(x)]=  \sum_{m=0}^{n} \binom{n}{m} (a)_{m}  f^{(a-m)}(x) x^{(n-m)}
\]
\end{proof}

Which thus concludes the proof of \ref{main formula 1}.

\subsection{The Fractional Leibniz product rule of an arbitrary differentiable function and the confluent hypergeometric limit function with a positive integer order}

Let  ${}_0F_1(n+1; x)$  be the confluent hypergeometric limiting function of $n+1$, where  $n\in \mathbb Z_{+}$, defined as

\[
{}_0F_1(n+1; x) 
 = \sum_{k=0}^{\infty}{\frac{(n)!}{(n+k)!}{\frac{x^{k}}{k!}}}
\]

\begin{theorem}
\label{main formula 2}
For $n \in \mathbb Z_{+}$, $a \in \mathbb{R} \setminus \mathbb{Z}_{<0}$$
$

\[
(\partial_{x})^{a} [{}_0F_1(n+1; x) f(x)]
= 
\sum_{m=0}^{n} \binom{n}{m} (a)_{m} \frac{\Gamma{(n+1-m)}}{\Gamma{(n+1-m+a)}} f^{(m)}(x) {}_0F_1(n-m+1+a; x)
\]
\end{theorem}

\begin{proof}

To prove this theorem, The same strategy to prove Theorem \ref{main formula 1}, and the following observation regarding a property of the confluent hypergeometric limiting function.

\textbf{\textit{Observations} }

For $n \in \mathbb Z_{+}$ , $a \in \mathbb{R}$
\[
(\partial_{x})^{a} {}_0F_1(n+1; x)
= \frac{\Gamma{(n+1)}}{\Gamma{(n+1+a)}}{}_0F_1(n+1-a; x)
\]

The proposition may be determined by the following observation:

note that:

\begin{align*}
    (\partial_{x}) {}_0F_1(n+1;x)&= \frac{{}_0F_1(n+2; x)}{(n+1)} \\
    (\partial_{x})^{2} {}_0F_1(n+1;x)&= \frac{{}_0F_1(n+3; x)}{(n+1)(n+2)}\\
    (\partial_{x})^{3}{}_0F_1(n+1;x)&= \frac{{}_0F_1(n+4; x)}{(n+1)(n+2)(n+3)} \\
   (\partial_{x})^{4}{}_0F_1(n+1;x)&= \frac{{}_0F_1(n+5; x)}{(n+1)(n+2)(n+3)(n+4)}\\
   \end{align*}
And in general, for any arbitary value $a$, it can be obtained that
   
 \begin{align*}  
(\partial_{x})^{a} {}_0F_1(n+1;x)&= \frac{ (n)!}{(n+a)!} {}_0F_1(n+1+a; x)\\
\end{align*}

Finally, $a$ is treated as a generic real number is used to express the  factorial as the gamma function to obtain the more general identity \citep{frac}:

\begin{equation}
    (\partial_{x})^{a} {}_0F_1(n+1; x)
= \frac{\Gamma{(n+1)}}{\Gamma{(n+1+a)}}{}_0F_1(n+1+a; x)
\end{equation}

\textbf{\textit{Proof of Theorem 3.2} }

We can now proceed with the proof of theorem \ref{main formula 2}.

The General leibiniz product rule for 
 ${}_0F_1(n+1; x) f(x)$ and its extension to non-integers is considered:

\[ (\partial_{x})^{a} [{}_0F_1(n+1; x)f(x)]=   \sum_{m=0}^{\infty} \binom{a}{m}  f^{(m)}(x)   (\partial_{x})^{a-m}[{}_0F_1(n+1; x)]
\]

and making use of equation (1), we can express a closed form expression of $(\partial_{x})^{a-m}{}_0F_1(n+1; x)$ to obtain:

\[
\sum_{m=0}^{\infty} \binom{a}{m}  f^{(m)}(x)   (n)_{m}(\partial_{x})^{a}[{}_0F_1(n+1-m; x)]
\]
\[
= \sum_{m=0}^{\infty} \binom{a}{m}  f^{(m)}(x)   (n)_{m} \frac{\Gamma{(n+1-m)}}{\Gamma{(n+1-m+a)}}{}_0F_1(n+1-m+a; x)
\]

Which, interchanging the total number elements of the binomial coefficient $\binom{a}{m}$ from $a$ to $n$ results in:

\[
\sum_{m=0}^{\infty} \binom{n}{m} (a)_{m} \frac{\Gamma{(n+1-m)}}{\Gamma{(n+1-m+a)}} f^{(m)}(x) {}_0F_1(n-m+1+a; x)
\]

Finally, truncating the infinite sum to a finite upper-index $n$, results in theorem \ref{main formula 2}. 

\end{proof}
This concludes our proof.

\section{An Umbral extension of the Truncated General Leibniz Product rule }

In this section, we provide a generalization of the truncated general Leibniz product rule for any linear operator and the product of an arbitrary differentiable function and any Sheffer sequence, making use of the theory of Umbral Calculus. We make use of the generalization to provide analogs of theorem \ref{main formula 1} for some Sheffer sequences, and provide proofs for expressions of the $m$th derivative of certain functions.

\vspace{10pt}
Let $O(f(t))$ indicates the smallest interger $n$ of the non-zero power series $f(t)$, such that the coefficient of $t^{n}$ does not vanish.
Let $O(k(t))=1$ and $O(g(t))=0$, then it can be shown that all Sheffer sequences can be characterized by the following generating function (See: \citep{Roman theory}, \citep{sheffer roman}, \citep{sheffer roman 2}) :

\[
   \frac{1}{g(k^{-1}(t))} e^{xk^{-1}(t)}  = \sum_{n=0}^{\infty} \frac{s_{n}(x) t^n}{n!}
\]  

Where $k^{-1}(t)$ is the compositional inverse of $k(t)$ such that $k(k^{-1}(t)) = k^{-1}(k(t)) = t$, and $s_{n}(x)$ is a polynomial sequence.

In particular, $s_{n}(x)$ is called the Sheffer sequence for $(g(t), k(t))$, if, for some linear operator characterized by $k(t)$ in the form $k(\partial_{x})$ the following relation holds:

\[
k(\partial_{x}) [s_{n}(x)] = n s_{n-1}(x) 
\]

The Generalized theorem will now be presented.
Let $s_{n}(x)$ be a sequence Sheffer to  $(g(t), k(t))$. Define a linear operator characterized by the power series $v(t)$ in the form $v(\partial_{x})$, and let $j(x)$ be an arbitrary function well defined for the operator $v(\partial_{x})$. Then the  generalized Leibniz product theorem is as follows:

\begin{theorem}
\label{Generalized theorem 1}
For $r=0, n \in \mathbb Z_{+}$
\[
    v(\partial_{x}) [s_{n}(x) j(x)] =
    \sum_{m=0}^{n} \binom{n}{m} s_{n-m}(x) (\partial_{r})^{m}[v(\partial_{x}+ k^{-1}(r)) j(x)]
\]

\end{theorem}

Where the operator $\partial_{x}$ is treated as a constant when $\partial_{r}$ is applied.

\begin{proof}
    
    First, consider
    
\[
v(\partial_{x}) [s_{n}(x) j(x)]
\] 

Which, by making use of the umbral  expression of $s_{n}(x)$ , it can be obtained

\[
v(\partial_{x}) [s_{n}(x) j(x)]
= v(\partial_{x}) (\partial_{r})^{n} \frac{1}{g(k^{-1}(r))} e^{xk^{-1}(r)}  j(x) =(\partial_{r})^{n} v(\partial_{x})\frac{1}{g(k^{-1}(r))} e^{xk^{-1}(r)}  j(x) \] 

Now, assume that the operator $v(\partial_{x})$ can be characterized by its power series expression:

\[
v(\partial_{x}) 
= \sum_{m=0}^{\infty}{A_{m} (\partial_{x})^{m}}
\]

Making use of the exponential shift theorem, where $q$ and $a$ is an arbitrary number:

\[ (\partial_{x})^{q} e^{ax} g(x) = e^{ax} (\partial_{x} + a)^{q} g(x)
\] 

It can be obtained that:

\[ 
 (\partial_{r})^{n} v(\partial_{x})\frac{1}{g(k^{-1}(r))} e^{xk^{-1}(r)}  j(x) 
 \] 
 \[
 =  (\partial_{r})^{n} \frac{1}{g(k^{-1}(r))} 
 \sum_{m=0}^{\infty}{A_{m} (\partial_{x})^{m}} \frac{1}{g(k^{-1}(r))} e^{xk^{-1}(r)}  j(x)\] 
\[ 
 =
(\partial_{r})^{n} \frac{1}{g(k^{-1}(r))} 
 \sum_{m=0}^{\infty}{A_{m} e^{xk^{-1}(r)} (\partial_{x} + k^{-1}(r))^{m} j(x)}
\] 
 \[
=
(\partial_{r})^{n} \frac{1}{g(k^{-1}(r))} e^{xk^{-1}(r)}
 \sum_{m=0}^{\infty}{A_{m}  (\partial_{x} + k^{-1}(r))^{m} j(x)}
\] 
\[ 
 =
 (\partial_{r})^{n} \frac{1}{g(k^{-1}(r))} e^{xk^{-1}(r)} v(\partial_{x}+k^{-1}(r))j(x)
\]

As we are differentiating with respect to $r$, Theorem \ref{General Leibniz rule of differentiation} can be applied with respect to $r$ to obtain:

\[ 
 (\partial_{r})^{n} \frac{1}{g(k^{-1}(r))} e^{xk^{-1}(r)} v(\partial_{x}+k^{-1}(r))j(x)
 \]
 \[ 
 = \sum_{m=0}^{n} \binom{n}{m} (\partial_{r})^{n-m} [\frac{1}{g(k^{-1}(r))} e^{xk^{-1}(r)}](\partial_{r})^{m}[v(\partial_{x}+ k^{-1}(r)) j(x)]\]

Where $r=0$. This is equal to  Theorem \ref{Generalized theorem 1}, thus concluding the proof.

\end{proof}

As a Corollary, a special case can be obtained, for the product of an arbitrary differentiable function and the standard polynomial:

\begin{corollary}

\label{main formula 1 generalization}

\[
k(\partial_{x}) [x^{n} f(x)] = \sum_{m=0}^{n} {\binom{n}{m} x^{n-m}k^{(m)}(\partial_{x})[f(x)]}
\] 

\end{corollary}

Where, similar to Theorem \ref{Generalized theorem 1}, $k^{(m)}(\partial_{x})$ refers to the $m$th derivative of the function $k(t)$, then letting variable $t$ be the operator $\partial_{x}$.

\begin{proof}
We make use of Theorem \ref{main formula 1}, and assume that the power series $k(t)$ is in the form

\[
k(t) = \sum_{m=0}^{\infty}{\frac{C_{m}t^{m}}{{m!}}}
\] 

So that
\[
k(\partial_{x}) (x^{n} f(x)) =
\sum_{m=0}^{\infty}{\frac{C_{m}}{{m!}}}(\partial_{x})^{m} x^{n} f(x) 
= \sum_{m=0}^{\infty}{\frac{C_{m}}{{m!}} \sum_{u=0}^{n}{\binom{n}{u}(m)_{u} x^{n-u} f^{(m-u)}(x)}}
\] 
\[
=\sum_{u=0}^{n}{\binom{n}{u} x^{n-u} } \sum_{m=0}^{\infty}{\frac{C_{m} f^{(m-u)}(x) (m)_{u}}{{m!}}} = \sum_{u=0}^{n}{\binom{n}{u} x^{n-u} } \sum_{m=0}^{\infty}{\frac{C_{m} f^{(m-u)}(x) }{{(m-u)!}}}  \] 

Making note of the following property regarding $k(t)$:

\[
k^{(v)}(t) = \sum_{m=0}^{\infty}{\frac{C_{m}t^{(m-v)}}{{(m-v)!}}}
\] 

We can determine that

\[
\sum_{u=0}^{n}{\binom{n}{u} x^{n-u} } \sum_{m=0}^{\infty}{\frac{C_{m} f^{(m-u)}(x) }{{(m-u)!}}} = \sum_{u=0}^{n}{\binom{n}{u} x^{n-u} } k^{(u)}(\partial_{x}) f(x) \] 

Thus concluding our proof.

\end{proof}

\section{Applications of the Umbral extension of the General Leibniz Product rule for Some Sheffer Sequences}

In this section, we make use of Theorem \ref{Generalized theorem 1} to obtain the analogs of Theorem \ref{main formula 1} for certain Sheffer sequences. An abitrary Apell sequence, the falling factorial, the rising factorial, and the associated Laguerre polynomials are considered.

\subsection{ The Fractional derivative of the product of an Apell sequence and an arbitrary differentiable function}

Due to Theorem \ref{Generalized theorem 1}, it easily follows, for $Ap_{n}(x)$ is any Apell sequence, that:
\begin{theorem}

\label{example 1}
 For $n \in \mathbb{N}$, $a 
    \in \mathbb{R} $
\[ 
(\partial_{x})^{a} [Ap_{n}(x) f(x)]=  \sum_{m=0}^{n} \binom{n}{m}  Ap_{n-m}(x) (a)_{m}(\partial_{x})^{a-m}[f(x)]
\]
\end{theorem}

Where $Ap_{n}(x)$ is characterized by the generating function:

\[ 
   \frac{1}{g(t)} e^{xt}  = \sum_{n=0}^{\infty} \frac{Ap_{n}(x) t^n}{n!}
\]  

and has the  Sheffer relation:

\[ 
(\partial_{x}) Ap_{n}(x) =n Ap_{n-1}(x)
\]

\begin{proof}
The proof is rather straightforward. As an Apell sequence is Sheffer to $(g(t), t)$, using \ref{main formula 1 generalization}, it can be obtained that:

\[
 (\partial_{x})^{a} [Ap_{n}(x)f(x)] =  \sum_{m=0}^{n} \binom{n}{m}  Ap_{n-m}(x) (a)_{m} (\partial_{x} +r)^{a-m} [f(x)]
\]
which, setting $r=0$, allows us Theorem \ref{example 1} to be obtained.

\end{proof}

Due to Theorem \ref{example 1}, the following Corollary can be obtained, which expresses the fractional order of any Apell sequence.

\begin{corollary}

\label{fractional apell}
For $n \in 
\mathbb{N}$, $a 
    \in \mathbb{R},  n-a > 0$

\[
Ap_{n-a}(x) = 
\sum_{m=0}^{n} \binom{n}{m} \frac{(a)_{m}}{n!} 
\frac{\Gamma{(n-a+1})}{\Gamma(m-a+1)}   Ap_{n-m}(x) x^{m-a} 
\]

\end{corollary}

\begin{proof}
To prove this, Theorem \ref{example 1} is used, and the condition $f(x)= x^{0}$ is set. We then
make use of the Sheffer property of Apell sequences, and obtain an expression for the fractional order of an Apell sequence by computing its $a$th derivative and using the gamma function to express its fractional derivative.

\[
(\partial_{x})^{a} [Ap_{n}(x) x^{0}]=
\frac{n!}{\Gamma{(n-a+1})}Ap_{n-a}(x) 
=
\sum_{m=0}^{n} \binom{n}{m} (a)_{m} Ap_{n-m}(x) 
\frac{x^{m-a}}{\Gamma(m-a+1)} \]

Which, multiplying both sides by $\frac{\Gamma{(n-a+1})}{n!}$, allows us to obtain Corollary \ref{fractional apell}, thereby concluding the proof.
\end{proof}

As Varona and al.  establishes in  \citep{apell}, the special values of certain analytical  functions  are expressible as Apell Sequences. Theorem \ref{example 1}   may be of particular interest for evaluation of the fractional derivative of those analytical functions expressible as Apell sequences.

\subsection{The Fractional derivative of product of the Falling factorial and an arbitrary differentiable function }

Utilizing \ref{Generalized theorem 1}, an analog of Theorem \ref{main formula 1} with the falling factorial instead of the standard polynomial can be obtained.

\begin{theorem}
\label{falling lebniz}
    For $n \in \mathbb{N}$, $a 
    \in \mathbb{R}$
    
\[ 
(\partial_{x})^{a}[(x)_{n} f(x)]
=
\sum_{m=0}^{n}\sum_{u=0}^{m}{\binom{n}{m} S_{1}(m,u) (x)_{n-m}  (a)_{u} (\partial_{x})^{a-u}[f(x)] }
\]

\end{theorem}

To prove  
Theorem \ref{falling lebniz}, we make use of the following Lemma regarding the $n$th derivative of $f(\ln{(x)})$:

\begin{lemma}
\label{Lemma ln}
for $n \in \mathbb{N}$,
\[ (\partial_{x})^{n}
f(\ln{(x)})
= \sum_{m=0}^{n} S_{1}(n,m) f^{(m)}(\ln{(x)}){(x)^{-n}}
\] 
\end{lemma}
The reader may note that the principle of mathematical induction is enough to prove the Lemma. For the purposes of this paper however,  a direct application of Theorem \ref{main formula 1} and Corollary \ref{main formula 1 generalization} to provide a will be used to provide a combinatorial proof of the Lemma. 
\begin{proof}
    First,  consider the following equation:

\[ 
f(\partial_{a}) (\partial_{x})^{a} [x^{n} g(x)] 
\] 
Where $f(\partial_{a})$ can be seen as a linear operator that acts on $a$ on the  operator $(\partial_{x})^{a}$. It can be easily seen that the application of the linear operator $f(\partial_{a})$ to any exponential function $b^{a}$ results in $f (\partial_{a})[b^{a}]  = b^{a}  f(\ln{b}) $, where $b$ is treated as a constant, and $f$ is the function characterizing the linear operator $f(\partial_{a})$.

This thus allows one to obtain the following evaluation:

\[ 
f(\partial_{a}) (\partial_{x})^{a} [x^{n} g(x)] 
=
(\partial_{x})^{a} f(\ln{\partial_{x}}) [x^{n} g(x)] 
\]

Proceed with making use of Theorem \ref{main formula 1} and consider the following:

\[ 
f(\partial_{a}) (\partial_{x})^{a} [x^{n} g(x)] 
=
(\partial_{x})^{a} f(\ln{\partial_{x}}) [x^{n} g(x)]
\]
\[
=
\sum_{m=0}^{n} \binom{n}{m} x^{n-m} 
f(\partial_{a}) (a)_{m} (\partial_{x})^{a-m} g(x)
\]

Making use of the representation of the falling factorial through the Stirling numbers of the first kind:

\[ 
(x)_{n}=\sum_{m=0}^{n}{S_{1}(n,m) x^{m} }
\]

One can obtain, for the sum expression of $(\partial_{x})^{a} f(\ln{\partial_{x}}) [x^{n} g(x)]$ , that
\[ 
\sum_{m=0}^{n} \binom{n}{m} x^{n-m} 
f(\partial_{a}) (a)_{m} (\partial_{x})^{a-m} g(x)
\]
\[ 
= 
\sum_{m=0}^{n} \sum_{v=0}^{m}\binom{n}{m} x^{n-m} S_{1}(m,v) 
f(\partial_{a}) [a^{v}(\partial_{x})^{a-m} g(x)]
\]

 Proceed with evaluating $f(\partial_{a}) [a^{v}(\partial_{x})^{a-m} g(x)]$, which can be done by making use of  Corollary \ref{main formula 1 generalization}. This thus allows one to get that:

\[ f(\partial_{a}) [a^{v} (\partial_{x})^{a-m} g(x)] = \sum_{r=0}^{v} \binom{v}{r} a^{v-r} f^{(r)}(\partial_{a})[(\partial_{x})^{a-m} g(x)]
\]

where the operator $\partial_{x}$ is treated  as a constant when $f(\partial_{a}$ is applied.

As a result, we may obtain that:
\[ 
(\partial_{x})^{a} f(\ln{\partial_{x}}) [x^{n} g(x)]
= 
\sum_{m=0}^{n} \sum_{v=0}^{m}
\sum_{r=0}^{v}
\binom{n}{m}\binom{v}{r} x^{n-m} S_{1}(m,v) a^{v-r} f^{(r)}(\partial_{a})[(\partial_{x})^{a-m} g(x)]
\]

We can evaluate $f^{(r)}(\partial_{a})[(\partial_{x})^{a-m} g(x)]$ such that

\[ 
f^{(r)}(\partial_{a})[(\partial_{x})^{a-m} g(x)]
=(\partial_{x})^{a-m} f^{(r)}(\ln{(\partial_{x}})) g(x)
\]

Thus resulting in:
\[ 
\sum_{m=0}^{n} \sum_{v=0}^{m}
\sum_{r=0}^{v}
\binom{n}{m}\binom{v}{r} x^{n-m} S_{1}(m,v) a^{v-r} f^{(r)}(\partial_{a})[(\partial_{x})^{a-m} g(x)]
\]
\[ 
=
\sum_{m=0}^{n} \sum_{v=0}^{m}
\sum_{r=0}^{v}
\binom{n}{m}\binom{v}{r} x^{n-m} S_{1}(m,v) a^{v-r} (\partial_{x})^{a-m} f^{(r)}(\ln{(\partial_{x}})) g(x)
\]

Which, setting $a=0$, allows us to obtain a sum expression for  $f(\ln{\partial_{x}}) [x^{n} g(x)]$ :

\begin{equation}
 f(\ln{\partial_{x}}) [x^{n} g(x)]
= 
\sum_{m=0}^{n}
\sum_{v=0}^{m} 
\binom{n}{m} x^{n-m} S_{1}(m,v)  (\partial_{x})^{-m} f^{(v)}(\ln{(\partial_{x}})) g(x)
\end{equation}

Finally, due to Theorem \ref{Generalized theorem 1}
 an alternative expression 
for $f(\ln{\partial_{x}}) [x^{n} g(x)]$ can be obtained:

\begin{equation}
f(\ln{\partial_{x}}) [x^{n} g(x)]
=
\sum_{m=0}^{n} \binom{n}{m} x^{n-m} (\partial_{J})^{m}[f(\ln{(J)})]g(x)
\end{equation}
Where $J$ is set to the operator $\partial_{x}$ in $f(\ln{(J)})$ .

Finally, by comparing the coefficients of both equation (3) and (4), it can be obtained that
\[ 
(\partial_{J})^{m}[f(\ln{(J)})]g(x)
=
\sum_{v=0}^{m}  S_{1}(m,v)  (\partial_{x})^{-m} f^{(v)}(\ln{(\partial_{x}})) [g(x)]
\]

 If two linear operators are treated as their corresponding functions by letting $\partial_{x}$ to be some arbitrary variable instead, Lemma  \ref{Lemma ln} is obtained, thereby concluding the proof.
 \end{proof}

\textbf{\textit{Proof of  
Theorem \ref{falling lebniz}} }

\begin{proof}

We first make use of the generating function characterization of the falling factorial \citep{Roman theory}:

\[
e^{x\ln{(1+t)}} =
\sum_{m=0}^{\infty}{\frac{(x)_{m} t^{m}}{m!}}
   \]  

We then proceed by  utilizing  \ref{Lemma ln} to obtain  the closed form of the $m$th derivative of $(y + \ln{(1+r)})^{a}$ with respect to $r$, where we treat $y$ as an arbitrary constant, and $a$ as a generic real number:
\[ 
(\partial_{r})^{m} [(y+\ln{(1+r)})^a]
= \sum_{u=0}^{m}{\frac{S_{1}(m,u) (a)_{u} (y+\ln{(1+r)})^{a-u}}{{(1+r)}^{m}}}
\]  

According to  \ref{Generalized theorem 1}, the fractional derivative of $(x)_{n} f(x)$ is 

\[
      (\partial_{x})^{a} [(x)_{n} f(x)] =
    \sum_{m=0}^{n} \binom{n}{m} (x)_{n-m} \sum_{u=0}^{m}{\frac{S_{1}(m,u) (a)_{u} (\partial_{x}+\ln{(1+r)})^{a-u}[f(x)]}{{(1+r)}^{m}}}
\] 
    which, if we evaluate at $r=0$ allows us to obtain the final evaluation

\[
(\partial_{x})^{a} [(x)_{n} f(x)] =
    \sum_{m=0}^{n} \binom{n}{m} (x)_{n-m} \sum_{u=0}^{m}{{S_{1}(m,u) (a)_{u} (\partial_{x})^{a-u}[f(x)]}}
    \] 
  \[  
  =
\sum_{m=0}^{n}\sum_{u=0}^{m}{\binom{n}{m} S_{1}(m,u) (x)_{n-m} (a)_{u} (\partial_{x})^{a-u}[f(x)] }
\] 

thus concluding our proof. 
\end{proof}

As a corollary, we can obtain a closed form expression of the fractional derivative of the falling factorial.
\begin{corollary}
     For $n \in \mathbb{N}$, $a 
    \in \mathbb{R}, a > 0$
    
\[ 
(\partial_{x})^{a}[(x)_{n}]
=
\sum_{m=0}^{n}\sum_{u=0}^{m}{\binom{n}{m} S_{1}(m,u)(x)_{n-m}  (a)_{u} \frac{x^{u-a}}{\Gamma(u-a+1)} }
\]      
\end{corollary}
Which can be easily derived by letting $f(x)=x^{0}$ in \ref{falling lebniz} and using the gamma function to express the fractional derivative of $x^{0}$.

\subsubsection{A Generalization for the iterated logarithm}
As a generalization, two propositions are presented. 
The first proposition states a closed form expression for the $m$th derivative of $f(\ln^{[\epsilon]}(x))$, where $f(\ln^{[\epsilon]}(x)) = f(\ln({\ln({\ln {\dots\ln{(x)})}) })}\dots $ is the $\epsilon$th iteration of $\ln({x})$ to the function $f$.

The second proposition provides an expression for the fractional derivative of the product of the associated Sheffer sequence $(x)^{[\epsilon]}_{n}$ (denoted as the iterated falling factorial) with $f(x)$, where 
 ${^{[\epsilon-1]}e} =
\underbrace {e^{e^{\cdot ^{\cdot ^{e}}}}} _{\epsilon-1}$ and the iterated falling factorial  is characterized by the generating function:
\[
 e^{x\ln^{[\epsilon]}{({^{[\epsilon-1]}e}+t)}} = \sum_{n=0}^{\infty} \frac{(x)^{[\epsilon]}_{n} t^n}{n!}
\]

\textbf{\textit{Proposition 1 } }

For $m, \epsilon \in \mathbb Z_{+}$,

\[
(\partial_{x})^{m} 
f(\ln^{[\epsilon]}(x))
=
\sum_{a_{1}=1}^{m}\sum_{a_{2}=1}^{a_{1}} \dots
\sum_{a_{\epsilon}=1}^{a_{\epsilon -1}} \frac{S_{1}(m,a_{1}) S_{1}(a_{1},a_{2})  \dots
S_{1}(a_{\epsilon -1},a_{\epsilon}) f^{(a_{\epsilon})}(\ln^{[\epsilon]}(x))}{x^{m} (\ln{(x)})^{a_{1}}(\ln^{[2]}{(x)})^{a_{2}}\dots(\ln^{[\epsilon-1]}{(x)})^{a_{\epsilon-1}}}
\]

\textbf{\textit{Proposition 2 } }

Due to Theorem \ref{Generalized theorem 1} and Proposition 1, one may have the following:

    For $n , \epsilon \in \mathbb{N}$, $\alpha 
    \in \mathbb{R}$
    
\[ 
(\partial_{x})^{\alpha}[(x)^{[\epsilon]}_{n} f(x)] 
=
\sum_{m=0}^{n}\sum_{a_{1}=1}^{m} \dots
\sum_{a_{\epsilon}=1}^{a_{\epsilon -1}}{\binom{n}{m}(x)^{[\epsilon]}_{n-m}(\alpha)_{a_{\epsilon}} \frac{S_{1}(m,a_{1})  \dots
S_{1}(a_{\epsilon -1},a_{\epsilon}) (\partial_{x})^{\alpha-a_{\epsilon}} [f(x)]}{({^{[\epsilon-1]}e})^{m} ({^{[\epsilon-2]}e})^{a_{1}}\dots(e)^{a_{\epsilon-2}}}}
\]    
    
\textbf{\textit{Sketch of Proof} }

Proposition A can be derived by continually applying the proof  strategy of Lemma \ref{Lemma ln} to obtain the $m$th derivative of $f(\ln^{[\epsilon]}{(x)}$. The principle of mathematical induction may be a possible strategy to prove said proposition. As a consequence, proving proposition A automatically allows a proof of  proposition B using Theorem \ref{Generalized theorem 1}.

\subsection{The Fractional derivative of product of the Rising factorial and an arbitrary differentiable function }

 Conversely, if the  rising factorial $x^{(n)}$ is considered, then a fractional Leibniz product rule can be obtained :

\begin{theorem}
    
\label{fractional rising}
     for $ n \in \mathbb{N}, a \in \mathbb{R}$

\[
(\partial_{x})^{a} [x^{(n)} f(x)] =
    \sum_{m=0}^{n} \sum_{u=0}^{m} \binom{n}{m} S_{1}(m,u)  x^{(n-m)} (-1)^{u}(a)_{u}(\partial_{x})^{a-u}[f(x)]
\] 
\end{theorem}

\begin{proof}
Making use of the generating function characterization of the rising factorial \citep{Roman theory}:

\[
e^{-x\ln{(1-t)}} =
\sum_{m=0}^{\infty}{\frac{x^{(m)} t^{m}}{m!}}
   \]  

   Tt can be immediately inferred, using  Theorem \ref{Generalized theorem 1} and  Lemma \ref{Lemma ln}, that the fractional Leibniz product rule must be in the form:

\[
(\partial_{x})^{a} [x^{(n)} f(x)] =
    \sum_{m=0}^{n} \binom{n}{m} x^{(n-m)} \sum_{u=0}^{m}{\frac{S_{1}(m,u) (a)_{u} (-1)^{u} (\partial_{x}-\ln{(1-r)})^{a-u}[f(x)]}{{(1-r)}^{m}}}
 \]  

 Which, if $r=0$,
allows us to obtain \ref{fractional rising}, thus concluding the proof.
\end{proof}

\begin{corollary}
     For $n \in \mathbb{N}$, $a 
    \in \mathbb{R}, a> 0 $
    
\[ 
(\partial_{x})^{a}[x^{(n)}]
=
\sum_{m=0}^{n}\sum_{u=0}^{m}{\binom{n}{m} S_{1}(m,u)x^{(n-m)}  (a)_{u} (-1)^{u} \frac{x^{u-a}}{\Gamma(u-a+1)} }
\]      
\end{corollary}
The Corollary can be easily derived by letting $f(x) = x^{0}$ in \ref{fractional rising} and using the gamma function to express the fractional derivative of $x^{0}$.

\subsection{The Fractional derivative of the product of the Exponential polynomials and an arbitrary differentiable function}

One can also construct an analog of Theorem \ref{main formula 1} with the Exponential polynomials instead of the standard polynomials, where the Exponential polynomials are defined by the following expression in terms of the Stirling numbers of the second kind:

\[
\phi_{n}(x) =\sum_{m=0}^{n}{S_{2}(n,m) x^{m} }
 \]  
and are characterized by the following generating function \citep{Roman theory}:

\[
 \sum_{m=0}^{\infty}{\frac{\phi_{m}(x) t^{m}}{m!}}=e^{x(e^{t}-1)}
  \]  

To do so, we will make use of the 
the following Lemma regarding the closed form of the $m$th derivative of $f(e^{x})$:

\begin{lemma}
\label{lemma exp}
For $m \in \mathbb{N}$
    \[
(\partial_{x})^{m} [f(e^{x})]
= \sum_{v=0}^{m}{S_{2}(m,v) f^{(v)}(e^x) (e^{x})^{v}}
    \]  
\end{lemma}

\textbf{\textit{Sketch of Proof of Lemma 4.8} }

The reader may note again that the principle of mathematical induction  is enough to prove the Lemma above. The same proof strategy of Lemma \ref{Lemma ln} can also be used to  prove the Lemma \ref{lemma exp}.

\begin{theorem}
\label{fractional bell}
    For $n \in \mathbb{N}$, $a 
    \in \mathbb{R}$

\[
(\partial_{x})^{a}[\phi_{n}(x)f(x)]=
\sum_{m=0}^{n} \sum_{u=0}^{m}
\binom{n}{m}  S_{2}(m,u) \phi_{n-m}(x) (a)_{u} (\partial_{x})^{a-u} [f(x)]
\]   
\end{theorem}

\begin{proof}

 Lemma  \ref{lemma exp} is used to obtain the closed form of the $m$th derivative of $(y+e^{r}-1)^{a}$ with respect to $r$, where we treat $y$ as an arbitrary constant, and $a$ as a generic real number:

\[ 
(\partial_{r})^{m}
[(y+e^{r}-1)^{a}]
= \sum_{u=0}^{m} S_{2}(m,u) (a)_{u} (y+e^{r}-1)^{a-u} e^{ur}
\]

Due to Theorem \ref{Generalized theorem 1}, the fractional derivative of $\phi_{n}(x) f(x)$ is:

\[ 
(\partial_{x})^{a}[\phi_{n}(x)f(x)]
=
\sum_{m=0}^{n}{\binom{n}{m} \phi_{n-m}(x) \sum_{u=0}^{m} S_{2}(m,u) (a)_{u} (\partial_{x}+e^{r}-1)^{a-u} e^{ur}} [f(x)]
\]

Which setting $r=0$ results in:

\[ 
(\partial_{x})^{a}[\phi_{n}(x)f(x)]
=
\sum_{m=0}^{n}{\binom{n}{m} \phi_{n-m}(x)\sum_{u=0}^{m} S_{2}(m,u) (a)_{u} (\partial_{x})^{a-u}} [f(x)]
\]

\[ 
=  
\sum_{m=0}^{n} \sum_{u=0}^{m}
\binom{n}{m}  S_{2}(m,u)\phi_{n-m}(x) (a)_{u} (\partial_{x})^{a-u} [f(x)]
\]
Which concludes the proof.

\end{proof}

\begin{corollary}
 For $n \in \mathbb{N}$, $a 
    \in \mathbb{R}, a > 0$
\[ (\partial_{x})^{a}[\phi_{n}(x)]=
    \sum_{m=0}^{n} \sum_{u=0}^{m}
\binom{n}{m}  S_{2}(m,u)\phi_{n-m}(x) (a)_{u} \frac{x^{u-a}}{\Gamma(u-a+1)}
  \]  
\end{corollary}

The Corollary can also be easily derived by letting $f(x) = x^{0}$ in Theorem \ref{fractional bell} and using the gamma function to express the fractional derivative of $x^{0}$.

\subsection{The Fractional derivative of the product of the  Associated Laguerre polynomials   and any arbitrary differentiable function}

  The Associated Laguerre polynomials, a class of Sheffer polynomials, have been widely applied in the fields of quantum physics and mathematical physics. It is, for instance, particularly useful in determining a hydrogen atom's wave function properties (See: \citep{laguerre 2} and 
\citep{laguerre 1}).

This section provides the  final  application of  \ref{Generalized theorem 1} for the Associated Laguerre polynomials $L_{n}^{\beta}(x)$, defined as being  Sheffer to $((1-t)^{-\beta-1}, \frac{t}{t-1})$ and are characterized by the generating function \citep{Roman theory}:

\[ 
\frac{1}{(1-t)^{\beta+1}} e^{x(\frac{t}{t-1})}
= \sum_{m=0}^{\infty}{\frac{L_{m}^{\beta}(x) t^{m}}{m!}}
\]  

So that  an Analog of \ref{main formula 1} with the Associated Laguerre polynomials can be obtained.

The following Lemma regarding the closed form of the $m$th derivative of $f( \frac{1}{x} )$ will uitilized to construct the analog:

\begin{lemma}
\label{lemma reciprocal}
For $m \in \mathbb{N}$
\[
(\partial_{x})^{m} f( \frac{1}{x} )
= \sum_{k=1}^{m} \binom{m-1}{k-1}\frac{m!}{k!} {\frac{(-1)^{m}}{x^{m+k}}  f^{(k)}(\frac{1}{x})}
    \]   
\end{lemma}

\textbf{\textit{Sketch of Proof of Lemma 4.11} }

Again, the reader may note that the principle of mathematical induction is enough to prove the Lemma above. The same proof strategy of  Lemma \ref{Lemma ln} and Lemma \ref{lemma exp} can also be made use of to prove the Lemma \ref{lemma reciprocal}

\begin{theorem}
\label{fractional Laguerre}
For $n \in 
\mathbb{N}$, $a 
    \in \mathbb{R}$
\[ 
(\partial_{x})^{a}[L_{n}^{\beta}(x)f(x)]=
\sum_{m=0}^{n} \sum_{k=1}^{m}
{\binom{n}{m}\binom{m-1}{k-1}\frac{m!}{k!}L_{n-m}^{\beta}(x)(-1)^{k} (a)_{k} (\partial_{x})^{a-k}[f(x)]}
 \]  
\end{theorem}

\begin{proof}
Lemma  \ref{lemma reciprocal} will made use of,  to obtain the closed form of the $m$th derivative of $(y+ \frac{1}{r-1})^{a}$ with respect to $r$, where we treat $y$ as an arbitrary constant, and $a$ as a generic real number:

\[ 
(\partial_{r})^{m} (y+ \frac{1}{r-1})^{a}
=
\sum_{k=1}^{m}{}
\binom{m-1}{{k-1}}\frac{m!}{k!} \frac{(-1)^{m}}{(r-1)^{m+k}}  (a)_{k}(y+ \frac{1}{r-1})^{a-k}
\]

Due to  Theorem \ref{Generalized theorem 1}, the fractional derivative of $L_{n}^{\beta}(x)f(x)$ is:

\[ 
(\partial_{x})^{a}[L_{n}^{\beta}(x)f(x)]
\] 
\[ 
=
\sum_{m=0}^{n}{\binom{n}{m}L_{n-m}^{\beta}(x)\sum_{k=1}^{m}{\binom{m-1}{{k-1}}\frac{m!}{k!} \frac{(-1)^{m}}{(r-1)^{m+k}}  (a)_{k}(\partial_{x}+ \frac{r}{r-1})^{a-k} [f(x)]}}
\] 

Which setting $r=0$ results in :

\[ 
(\partial_{x})^{a}[L_{n}^{\beta}(x)f(x)]
=
\sum_{m=0}^{n}{\binom{n}{m}L_{n-m}^{\beta}(x)\sum_{k=1}^{m}{\binom{m-1}{{k-1}}\frac{m!}{k!} \frac{(-1)^{m}}{(-1)^{m+k}}  (a)_{k}(\partial_{x})^{a-k} [f(x)]}}
\] 
\[ 
=
\sum_{m=0}^{n} \sum_{k=1}^{m}
{\binom{n}{m}\binom{m-1}{k-1}\frac{m!}{k!}L_{n-m}^{\beta}(x)(-1)^{k} (a)_{k} (\partial_{x})^{a-k}[f(x)]}
 \]  

 Thereby concluding the proof.

 \begin{corollary}
     For $n \in \mathbb{N}$, $a 
    \in \mathbb{R}, a>0 $
\[ 
(\partial_{x})^{a}[L_{n}^{\beta}(x)]
=
\sum_{m=0}^{n} \sum_{k=1}^{m}
{\binom{n}{m}\binom{m-1}{k-1}\frac{m!}{k!}L_{n-m}^{\beta}(x)(-1)^{k} (a)_{k} \frac{x^{k-a}}{\Gamma(k-a+1)}}
\]  
 \end{corollary}

The Corollary can also be easily obtained by setting $f(x)= x^{0}$ in Theorem \ref{fractional Laguerre} and using the gamma function to express the fractional derivative of $x^{0}$.

\end{proof}

\section{Discussion}
This paper highlights the seminal Leibiniz product theorem and makes use of the theory of Umbral calculus to extend it to fractional indices for certain functions.
More specifically, the paper provides a generalized Leibiniz product rule for Sheffer sequences to obtain the fractional derivative of the product of certain Sheffer sequences and an arbitrary differentiable function. This paper also investigates a fractional Leibiniz product rule for the product of an arbitrary differentiable equation and the confluent hypergeometric limiting function for positive interger order.

\vspace{10pt}
Further investigations are encouraged to investigate the potential applications of these theorems; specifically, the field of fractional calculus and its physical applications.


\begin{thebibliography}{99}
\bibitem{hypergeometric 2}
V. Kiryakova. Fractional calculus operators of special functions? The result is well predictable!. Chaos, Solitons Fractals, 102, 2-15.
\end{thebibliography}
\end{document}